\documentclass[11pt]{article}\usepackage{epsfig}\usepackage{setspace}
\usepackage{graphicx,psfrag}
\usepackage{amssymb}\usepackage{amsmath}\usepackage{amsthm}\usepackage{makeidx}
\usepackage[T1]{fontenc}
\author{Harry Crane}
\title{Infinitely exchangeable random graphs generated from a Poisson point process on monotone sets and applications to cluster analysis for networks}
\date{}

\def\sup{\mathop{\rm sup}\nolimits}

\def\dom{\mathop{\rm dom}\nolimits}
\def\cod{\mathop{\rm cod}\nolimits}
\def\pr{\mathop{\rm pr}\nolimits}
\def\rest{\mathop{\bf rest}\nolimits}
\def\dr{\mathop{\bf dr}\nolimits}
\def\catC{\mathop{\bf C}\nolimits}
\def\catD{\mathop{\bf D}\nolimits}
\def\cate{\mathop{\mathcal{E}}\nolimits}
\def\catf{\mathop{\mathcal{F}}\nolimits}
\def\catg{\mathop{\mathcal{G}}\nolimits}
\newtheorem{thm}{Theorem}[section]
\newtheorem{lemma}[thm]{Lemma}

\newtheorem{cor}[thm]{Corollary}

\makeindex
\begin{document}
\maketitle
\begin{abstract}
We construct an infinitely exchangeable process on the set $\cate$ of subsets of the power set of the natural numbers $\mathbb{N}$ via a Poisson point process with mean measure $\Lambda$ on the power set of $\mathbb{N}$.  Each $E\in\cate$ has a least monotone cover in $\catf$, the collection of monotone subsets of $\cate$, and every monotone subset maps to an undirected graph $G\in\catg$, the space of undirected graphs with vertex set $\mathbb{N}$.  We show a natural mapping $\cate\rightarrow\catf\rightarrow\catg$ which induces an infinitely exchangeable measure on the projective system $\catg^{\rest}$ of graphs $\catg$ under permutation and restriction mappings given an infinitely exchangeable family of measures on the projective system $\cate^{\rest}$ of subsets with permutation and restriction maps.  We show potential connections of this process to applications in cluster analysis, machine learning, classification and Bayesian inference.
\end{abstract}
\section{Introduction}\label{section:introduction}
Here, we show a construction of an infinitely exchangeable family of random graphs which is based on an associated Poisson point process on the power set of the natural numbers $\mathbb{N}$.  We provide a necessary and sufficient condition for the induced random graph to be infinitely exchangeable and discuss a potential use for this model in the area of cluster analysis and stochastic classification, which has been previously studied in a statistical and machine learning context in previous work by Jordan \cite{BleiJordanNg2004} (with Blei and Ng) and \cite{BroderickJordanPitman2011} (with Broderick and Pitman) McCullagh \cite{McCullaghYang2006,McCullaghYang2008}, but outside of the realm of network analysis.  

We now introduce preliminary material and notation which is critical to our treatment.
\subsection{Projective systems}\label{section:projective systems}
A {\em projective system} associates with each finite set $[n]$ a set $Q_n$ and with each one-to-one injective map $\varphi:[m]\rightarrow[n]$, $m\leq n$, a projection $\varphi^*:Q_n\rightarrow Q_m$ which maps $Q_n$ into $Q_m$ such that 
\begin{itemize}
	\item if $\varphi$ is the identity $[n]\rightarrow[n]$ then $\varphi^*$ is the identity $Q_n\rightarrow Q_n$ and
	\item if $\psi:[l]\rightarrow[m]$, $l\leq m$, and $\psi^*:Q_m\rightarrow Q_l$ is its associated projection, then the composition $(\varphi\psi):[l]\rightarrow[n]$ satisfies $(\varphi\psi)^*\equiv\psi^*\varphi^*:Q_n\rightarrow Q_l$.
\end{itemize}

If $Q_n$ is the set of subsets of $[n]^2$, i.e.\ the space of directed graphs with $n$ vertices, one can define the projection $Q_n\rightarrow Q_m$ either by {\em restriction} or {\em delete-and-repair}.  Each $A\in Q_n$ can be represented as an $n\times n$ matrix with entries in $\{0,1\}$ such that $A_{ij}=1$ if $(i,j)\in A$ and $A_{ij}=0$ otherwise.  For each $n\geq1$, let $\varphi_{n,n+1}$ be the operation on $Q_{n+1}$ which restricts $A$ to the complement of $\{n+1\}$.  In matrix form, $\varphi_{n,n+1}A=:A_{|[n]}$ is the $n\times n$ matrix obtained from $A$ by removing the last row and last column of $A$ and keeping the rest of the entries unchanged.  It is clear that the compositions $\varphi_{m,n}:=\varphi_{m,m+1}\circ\cdots\circ \varphi_{n-1,n}$ for $m\leq n$ are well-defined as the restriction of $A\in Q_n$ to $[m]$ by removing the last $n-m$ rows and columns of $A$.

For $n\geq1$, we write $\mathcal{S}_n$ to denote the symmetric group of permutations of $[n]$, i.e.\ one-to-one maps $[n]\rightarrow[n]$.  Each $\sigma\in\mathcal{S}_n$ acts on each element $A\in Q_n$ componentwise in the usual way.  That is, $(i,j)\in A$ if and only if $(\sigma(i),\sigma(j))\in\sigma(A)$.  The restriction maps $(\varphi_{m,n},m\leq n)$ together with permutation maps $(\sigma\in\mathcal{S}_n,n\geq1)$ and their compositions make $Q:=(Q_n,n\geq1)$ a projective system.

Another way to specify a projective system on $(Q_n,n\geq1)$ is by {\em delete-and-repair}.  For $n\geq m\geq1$, let $\psi_m$ act on $A\in Q_n$ by removing the $m$th row and column of $A$ and directing an edge from each $i$ in $\{j\in[n]:(j,m)\in A\}$ to each $k$ in $\{j\in[n]:(m,j)\in A\}$.  In other words, $\psi_mA$ is obtained by deleting the vertex labeled $m$ from $A$ and connecting two vertices $i$ and $k$ by a directed edge from $i$ to $k$ if both $(i,m)$ and $(m,k)$ are elements of $A$, i.e.\ there is a directed path $i\rightarrow m\rightarrow k$ in $A$.

For $m\leq n$, define $\psi_{m,n}:=\psi_{m+1}\circ\cdots\circ\psi_{n}$.  Plainly, $\psi_{m,n}$ is well-defined since for each $n\geq2$, $\psi_{n-2,n}\equiv\psi_{n-1}\circ\psi_n=\psi_n\circ\psi_{n-1}$ and $\psi_{l,n}=\psi_{l,m}\circ\psi_{m,n}$.  The delete-and-repair maps $(\psi_{m,n},m\leq n)$ together with permutation maps $(\sigma\in\mathcal{S}_n,n\geq1)$ and compositions also make $(Q_n,n\geq1)$ a projective system, which differs from the above projective system based on restriction maps.

Throughout the rest of this paper, for a system $Q:=(Q_n,n\geq1)$ we write $\varphi_{m,n}^{Q}$ to represent the restriction maps for collection $Q$ and $\psi_{m,n}^{Q}$ to represent the corresponding delete-and-repair maps.  We write $Q^{\rest}$, resp.\ $Q^{\dr}$, to denote the projective system on $Q$ described by the restriction, resp.\ delete-and-repair, maps together with permutation maps.
\subsection{Power set of $2^{[n]}$}\label{section:power set}
Let $\cate_n=2^{2^{[n]}}$ denote the power set of all subsets of $2^{[n]}$, the power set of $[n]$.  For each $n\geq1$, we define the partial order $\leq_{\cate}$ on $\cate_n$ by
$$E\leq_{\cate} E'\Longleftrightarrow [e\in E\Rightarrow e\in E'].$$
Therefore, $\{\{1\}\}\leq_{\cate}\{\{1\},\{1,2\}\}$ but $\{\{1\}\}$ and $\{\{1,2\}\}$ have no relation under $\leq_{\cate}$.

We define restriction maps $\varphi_{m,n}^{\cate}:\cate_n\rightarrow\cate_m$ as follows.  For $E:=\{e_i:1\leq i\leq k\}\in\cate_n$, we define $\varphi_{m,n}^{\cate}(E):=E_{|[m]}:=\{e_i\cap[m]:1\leq i\leq k\}$.  The maps $(\varphi^{\cate}_{m,n},m\leq n)$ preserve $\leq_{\cate}$ on $\cate$ since 
$$E\leq_{\cate} E'\Longleftrightarrow [e\in E\Longrightarrow e\in E']$$
and so
$$ e\in E\Longrightarrow e\cap[m]\in \varphi_{m,n}^{\cate}E, e\cap[m]\in \varphi_{m,n}^{\cate}E'\Longrightarrow \varphi^{\cate}_{m,n}E\leq_{\cate}\varphi^{\cate}_{m,n}E'.$$
Permutations act on elements of $E$ componentwise, i.e.\ for $n\geq1, E\in\cate_n$, and $\sigma\in\mathcal{S}_n$, we have $\sigma(E):=\{\sigma(e):e\in E\}$.  We write $\cate^{\rest}$ to denote the projective system $\cate:=(\cate_n,n\geq1)$ together with restriction maps $(\varphi^{\cate}_{m,n},m\leq n)$ and permutation maps $(\sigma\in\mathcal{S}_n,n\geq1)$ and their compositions.

\subsection{Monotone sets}\label{section:monotone sets}
A subset $A\subset 2^{[n]}$ is \it monotone \rm if $a\in A$ implies $2^{a}\subset A$.  For example, the set $A=\{\emptyset,\{1\},\{2\},\{3\},\{1,2\}\}=\langle\{1,2\},\{3\}\rangle$ is a monotone set with maximal elements $\{1,2\}$ and $\{3\}$, which constitute the generating class of $A$, written $G(A).$  An element $a$ of the generating class of a monotone set $A$ is a maximal element of $A$ in the sense that no other element $a'\in A$ contains $a$ as a subset.  The generating class $G(A)$ of a monotone subset $A$ consists of all maximal elements of $A$.  We write $\mathcal{F}_n$ as the set of monotone sets taking elements in $2^{[n]}$.  Note that a monotone set is uniquely determined by its generating class, and so we will write $A\in\mathcal{F}_n$ and $G(A)$ to describe the same object, i.e.\ the monotone set $A$, whenever it is convenient to do so.  Every subset $E$ of $2^{[n]}$ has a least monotone cover in $\mathcal{F}_n$, which we denote by $\alpha(E)$ and is given by $\alpha(E):=\{2^a:a\in E\}$.

For each $n\geq1$, $\mathcal{F}_n$ is a partially ordered set induced by the partial order {\em inclusion} on $2^{[n]}$, i.e.\ for $A,B\in\mathcal{F}_n$, we say $A\leq_{\catf} B$ if and only if each $a\in G(A)$ is a subset of some $b\in G(B)$.  And for any pair $A,B$ with $A\leq B$, the intervals $[A,B], (A,B]$ and $[A,B)$ are well-defined subsets of $\mathcal{F}_n$.  Note that we intend the symbols $\subseteq$ and $\subset$ to have strictly different meanings in this paper.  In particular, we write $A\subseteq B$ to mean $A$ is any subset of $B$, while we write $A\subset B$ to mean $A$ is a {\em proper subset} of $B$, i.e.\ $A\subseteq B$ but $A\neq B$.  This distinction becomes important in the next section.

We define the operation {\em restriction} $\varphi^{\catf}_{n,n+1}:\mathcal{F}_{n+1}\rightarrow\mathcal{F}_n$ as the operation which maps $G(A)\mapsto G(A)\cap[n]$.  That is, for $A\in\mathcal{F}_{n+1}$ with $G(A):=\{A_1,\ldots, A_k\}$, $\varphi^{\catf}_{n,n+1}A:=\langle A_i\cap[n]:i=1,\ldots,k\rangle$.  For $m\leq n$, we define $\varphi^{\catf}_{m,n}:=\varphi^{\catf}_{m,m+1}\circ\cdots\circ\varphi^{\catf}_{n-1,n}$ in the usual way by composition, and the collection of restriction maps $\catf:=(\varphi^{\catf}_{m,n},m\leq n)$ together with permutation maps makes $(\mathcal{F}_n,n\geq1)$ into a projective system. written $\catf^{\rest}$.  Here, a permutation $\sigma\in\mathcal{S}_n$ acts on a monotone set $A\in\mathcal{F}_n$ by acting componentwise on its generating class, i.e.\ $a\in G(A)$ if and only if $\sigma(a)\in G(\sigma(A))$.

Also note that the inverse mapping $\varphi_{n,n+1}^{\catf^{-1}}$ associates with each $F\in\catf_n$ an interval, and also maps intervals in $\catf_n$ to intervals in $\catf_{n+1}$.  In particular, for $n\geq1$ and $A:=\langle A_1,\ldots, A_k\rangle\in\mathcal{F}_n$, we have
$$\varphi_{n,n+1}^{\catf^{-1}}(A)=\left[\langle G(A)\rangle, \left\langle A_j\cup\{n+1\}:j=1,\ldots,k\right\rangle\right].$$
\subsection{Undirected graphs}\label{section:undirected graphs}
For $n\geq1$, an {\em undirected graph} $G\in\mathcal{G}_n$ is a pair $(V,E)$ of vertices $V$ and edges $E$ whereby the number of vertices $\#V$ of $G$ is $n$, and without loss of generality we assume $V=[n]:=\{1,\ldots,n\}$, and the edges are a subset of $[n]^2$ such that $(i,j)\in E$ implies $(j,i)\in E$.

For each $n\geq1$, the elements of $\mathcal{G}_n$ correspond to the symmetric subsets of $[n]^2$, e.g.\ for $A\in\mathcal{G}_n$, $(i,j)\in A$ if and only if $(j,i)\in A$, and hence $\catg:=(\mathcal{G}_n,n\geq1)$ is a projective system under both restriction and delete-and-repair, as described in section \ref{section:projective systems}.  For each $n\geq1$, $\catg_n$ is a poset under $\leq_{\catg}$ where for $G,G'\in\catg_n$
$$G\leq_{\catg}G'\Longleftrightarrow [\{i,j\}\in G\Longrightarrow \{i,j\}\in G'].$$
For clarity of notation, write $\varphi_{m,n}^{\catg}$ and $\psi_{m,n}^{\catg}$ to denote the operations of restriction and delete-and-repair respectively on the projective system $\catg^{\rest}$ and $\catg^{\dr}$ respectively.
\subsection{Some category theory}\label{ref:category theory}
The relationship between the collections $\cate,\catf$, and $\catg$ is described in a straightforward way by elementary concepts in category theory \cite{AwodeyBook}.

A {\em category} {\bf C} consists of {\em objects} $A,B,C,\ldots$ and {\em arrows} $f,g,h,\ldots$ between objects so that for each arrow $f\in$ {\bf C} there are objects $\dom(f)$ and $\cod(f)$ in {\bf C}, the domain and codomain respectively of $f$, and we write $f:\dom(f)\rightarrow\cod(f)$.  Given $f,g\in$ {\bf C} such that $\cod(f)=\dom(g)$, the {\em composite} $g\circ f:\dom(f)\rightarrow\cod(g)$ is an arrow in {\bf C}.  Also, for each object $A\in$ {\bf C} there is an identity arrow $1_A:A\rightarrow A$, and all arrows of {\bf C} must satisfy associativity and preservation under composition with the identity functions.

In each of the categories we define below, there is at most one arrow between any two objects.  Therefore, if an arrow $f$ corresponds to the composition of arrows $g\circ h$, these must represent the same arrow.  Under this assumption, for a pair of objects we need not make explicit all of the various compositions of arrows which result in an arrow between these objects, as they are implicitly assumed to be there, and are all the same arrow.

We define three categories as follows.  We write $\mathcal{E}^{\rest}$ for the category with objects given by the elements of $\cate$ and arrows given by the restriction $(\varphi^{\cate}_{m,n},m\leq  n)$ and permutation $(\sigma\in\mathcal{S}_n,n\geq1)$ maps, and compositions of these maps.  That is, for $m\leq n$ and $E\in\cate_n$, there is an arrow $E\rightarrow\varphi_{m,n}^{\cate}E=E_{|[m]}$.  We denote the arrow between $E$ and $E_{|[m]}$ by $\varphi_{m,n}^{\cate}E$.  Likewise, we define $\catf^{\rest}$ and $\catg^{\rest}$ to be the category with objects given by the elements of $\catf$ and $\catg$ respectively, and arrows defined by the restriction maps, $(\varphi_{m,n}^{\catf},m\leq n)$ and $(\varphi_{m,n}^{\catg}, m\leq n)$ resp., and permutation maps.  For example, in $\catf^{\rest}$, we have $\langle\{1,2\},\{3\}\rangle\rightarrow_{\varphi_{2,3}^{\catf}}\langle\{1,2\}\rangle.$

A {\em functor} between categories {\bf C} and {\bf D} is a map $F:\mathbf{C}\rightarrow\mathbf{D}$ which takes objects in {\bf C} to objects in {\bf D} and arrows in {\bf C} to arrows in {\bf D}.  There are natural functors $\alpha:\cate^{\rest}\rightarrow\catf^{\rest}$ and $\beta:\catf^{\rest}\rightarrow\catg^{\rest}$, which we define as follows.
\begin{itemize}
\item[(i)] For each $n\geq1$ and $E\in\mathcal{E}_n$, $\alpha(E)\in\mathcal{F}_n$ is the least monotone cover of $E$.
\item[(ii)] For each $n\geq1$ and $F:=\langle F_1,\ldots,F_k\rangle\in\mathcal{F}_n$, $\beta(F)=G_F\in\mathcal{G}_n$ where $(i,j)\in G_F$ if and only if $\{i,j\}\subseteq F_j$ for some $j=1,\ldots,k$.
\item[(iii)] For $m\leq n$, $\alpha(\varphi^{\cate}_{m,n})=\varphi^{\catf}_{m,n}$ and $\beta(\varphi^{\catf}_{m,n})=\varphi^{\catg}_{m,n}$ so that $(\alpha\circ\beta)(\varphi_{m,n}^{\cate})=\varphi_{m,n}^{\catg}$.
	\item[(iv)] For every $n\geq1$, $\alpha(\sigma)=\sigma$ and $\beta(\sigma)=\sigma$ for all $\sigma\in\mathcal{S}_n$.
\end{itemize}
Functors between posets preserve partial ordering.  For $\bullet=\cate,\catf,\catg$, we have defined the partial order $\leq_{\bullet}$ in previous sections.  For $m\leq n$ and $A, A'\in\bullet_n$ such that $A\leq_{\bullet} A'$, it holds that $\varphi_{m,n}^{\bullet}A\leq_{\bullet}\varphi_{m,n}^{\bullet}A'$.  Hence, for each fixed $1\leq m\leq n$, the restriction maps $(\varphi_{m,n}^{\bullet})$ define a functor $\bullet_n\rightarrow\bullet_m$, where we take the elements of $\bullet_n$ and $\bullet_m$ as objects and the partial order $\leq_{\bullet}$ defines the arrows, i.e.\ there is an arrow $A\rightarrow A'$ in $\bullet_n$ if and only if $A\leq_{\bullet}A'$.

We can also regard the collections $\cate^{\rest},\catf^{\rest},\catg^{\rest}$ as partially ordered sets with partial order $\leq_{\bullet^{\rest}}$, $\bullet=\cate,\catf,\catg$, defined as follows.  For $m\leq n$, $A\in\bullet_n$, $A'\in\bullet_m$,
$$A'\leq_{\bullet^{\rest}}A\Longleftrightarrow A'=\varphi_{m,n}^{\bullet}A.$$
The functors $\alpha,\beta$ and $\beta\circ\alpha$ above preserve the partial orders $\leq_{\cate^{\rest}},\leq_{\catf^{\rest}}$ and $\leq_{\catg^{\rest}}$.

Below, we show a construction of infinitely exchangeable random elements of $\mathcal{E}^{\rest}$.  In particular, we construct an infinitely exchangeable random monotone set by projecting from a Poisson point process $X$ on the power set to a random subset $X^*\in\cate^{\rest}$ of the power set to the least monotone cover $\alpha(X^*)\in\catf^{\rest}$, which corresponds to a random graph $\beta\circ\alpha(X^*)\in\mathcal{G}^{\rest}$.  This procedure looks like this 
$$X\rightarrow_{*}X^*\rightarrow_{\alpha}\alpha(X^*)\rightarrow_{\beta}\beta\circ\alpha(X^*).$$

\paragraph{Projective systems in statistics}
The relevance of category theory and projective systems in statistical modeling is discussed in detail by McCullagh \cite{McCullagh2002}.  Here we have introduced projective systems for the collection of subsets of $2^{[n]}$, their associated monotone subsets of $2^{[n]}$ and their associated undirected graphs in $[n]^2$.  

The choice of projection, i.e.\ either restriction or delete-and-repair, on $\catg$ is intended to reflect the notion of subsampling in statistics, and each admits its own statistical interpretation in terms of subsampling which may be appropriate depending on the application, and the actual way in which observations are made.  In the study of directed graphs, in particular permutations, delete-and-repair is often used, see e.g.\ the Chinese restaurant process (CRP) on permutations \cite{PitmanCSP}.  However, for our purposes, restriction maps are a natural choice.

To be specific, if we think of a graph as modeling a social network, the edges of the graph represent social links among the agents (nodes) in the population.  When observing a social network, there are several sampling methods which are reasonable.  The most intuitive of these are perhaps node sampling and snowball sampling.  In node sampling, we first sample a set $\{u_1,\ldots,u_n\}$ of nodes and subsequently observe any edges between these nodes.  In this setting, restriction accurately reflects our sampling method since removal of a node from our sample, e.g.\ $u_n$, removes any edges which involves that node from our sampled graph.  In snowball sampling, we start with a set of nodes, without loss of generality suppose we start with one node $u$, and proceed as follows.  Given $U_0=u$, sample $U_1:=\{v\in\mathcal{U}:v\sim U_0\}$, the set of nodes $v$ adjacent to some node in $U_0$ in the graph we are observing.  Subsequently, put $U_{k+1}:=\{v\in\mathcal{U}:v\sim U_k,\mbox{ }v\notin U_k,U_{k-1},\ldots,U_0\}$.  In snowball sampling, we stop at some arbitrary level $k$.  This way of sampling results in a sample of all nodes at radius $k$ or less from $u$, and depends on our choice of $u$ and the network structure.  In this case, removal of a node $v\in V:=\bigcup_{i=0}^{k}U_i$ results in removal of all other nodes $w\in V$ such that $v$ lies along every path between $u$ and $w$.  The result of this type of subsampling is not described simply via restriction.  Below, we discuss only those projective systems characterized by the restriction maps, and the implications of this method of sampling on inference.
\subsection{Infinite exchangeability}\label{section:infinite exchangeability}
A family $(p_n,n\geq1)$ of probability measures on a projective system $Q:=(Q_n,n\geq1)$ is {\em infinitely exchangeable} if it is invariant under both the action of permutations, called {\em finite exchangeability}, and selection according to the projection maps $(D_{m,n},m\leq n)$ associated with the system, called {\em consistency}.  For example, a family of measures $(p_n,n\geq1)$ on $Q$ is infinitely exchangeable if
\begin{itemize}
	\item for each $n\geq1$ and $\sigma\in\mathcal{S}_n$, $p_n(A)=p_n(\sigma(A))$ for every $A\in{Q}_n$ and
	\item for every $m\leq n$, $p_m(A)=p_n(D_{m,n}^{-1}(A))$ for every $A\in{Q}_m$.
\end{itemize}
An infinitely exchangeable collection of measures uniquely characterizes a measure $p$ on the infinite space associated with $Q$ through its finite-dimensional distributions and invariance under projection and permutation maps.

Above we have defined functors $\alpha:\cate^{\rest}\rightarrow\catf^{\rest}$ and $\beta:\catf^{\rest}\rightarrow\catg^{\rest}$.  One property of functors is that $\alpha(E_1\circ E_2)=\alpha(E_1)\circ\alpha(E_2)$.  Hence, we have the following elementary lemma which is useful for our purposes below.

\begin{lemma}\label{lemma:functor inf exch} Let $\gamma:\catC\rightarrow\catD$ be a functor between categories $\catC$ and $\catD$ where the objects $C$ of $\catC$ and $D$ of $\catD$ form a projective system under mappings $\pi^{\catC}$ and $\pi^{\catD}$ respectively, and the arrows are defined by the partial ordering induced by the projections $\pi^{\catC}$ and $\pi^{\catD}$.  Let $\mu_{\catC}$ be a probability measure on $C$ and let $\mu_{\catD}=\mu_{\catC}\gamma^{-1}$ be the distribution induced on the objects of $\catD$ by $\mu_{\catC}$ through the functor $\gamma$.  If $\mu_{\catC}$ is invariant under $\pi^{\catC}$, i.e.\ $\mu_{\catC}=\mu_{\catC}\pi^{\catC^{-1}}$, then $\mu_{\catD}$ is invariant under $\pi^{\catD}$.
\end{lemma}
\begin{proof}
Since $\gamma:\catC\rightarrow\catD$ is a functor, we have $\gamma(g\circ f)=\gamma(g)\circ\gamma(f)$ for any arrows $f,g\in\catC$ such that $\dom(g)=\cod(f)$.  Let $c$ be an object in $\catC$ and $c':=\pi^{\catC}(c)$, so that we have an arrow $c\rightarrow_{\pi^{\catC}(c)}c'$ which we denote by $\pi^{\catC}(c)$, and the action of the functor $\gamma$ is such that
$$(\gamma\circ\pi^{\catC})(c)=\gamma(\pi^{\catC}(c))=\pi^{\catD}(\gamma(c'))=(\pi^{\catD}\circ\gamma)(c'),$$
and we have $\gamma\circ\pi^{\catC}=\pi^{\catD}\circ\gamma$ and invariance of the induced measure $\mu_{\catD}$ follows.  Indeed, for $d=\gamma(c)$ and $d'=\gamma(c')$,
$$\mu_{\catD}(d')\equiv\mu_{\catC}\gamma^{-1}(d')=(\mu_{\catC}\pi^{\catC^{-1}})\gamma^{-1}(d')=\mu_{\catC}(\gamma\circ\pi^{\catC})^{-1}(d')=\mu_{\catC}(\pi^{\catD}\circ\gamma)^{-1}(d')=(\mu_{\catC}\gamma^{-1})\pi^{\catD^{-1}}(d')=\mu_{\catD}\pi^{\catD^{-1}}(d').$$
\end{proof}
This will simplify our proofs below for infinite exchangeability of the random graph induced by a random subset of $2^{[n]}$.
\section{Construction of an infinitely exchangeable random graph}\label{section:construction}
Let $X$ be a Poisson point process on $2^{[n]}$ with mean measure $\Lambda$ so that $\{X_a:a\in 2^{[n]}$\} is a collection of independent Poisson random variables with each $X_a$ having mean $\Lambda(a)\geq0$.  By ignoring multiplicities, each realization of this process defines a random subset $X^*:=\{a\subseteq[n]:X_a>0\}\in\cate_n$ which consists of those points $a\in2^{[n]}$ for which $X_a>0$.  The distribution of $X^*$ is given by
$$\mathbb{P}(X^*\subseteq E)=\prod_{a\subset[n]:a\notin E}\exp\{-\Lambda(a)\}.$$

  In general, $X^*$ will not be monotone, but, as discussed above, it will have a least monotone cover given by $\alpha(X^*):=\{2^a:a\in X^*\}\in\catf_n$.

It is straightforward to compute the induced distribution of $\alpha(X^*)$ on $\mathcal{F}_n$ under the partial ordering $\leq_{\catf}$.  That is, for $A\in\mathcal{F}_n$, let $\bar{A}$ denote the complement of $A$ in $2^{[n]}$, then
\begin{eqnarray*}
\mathbb{P}_n(\alpha(X^*)\leq_{\catf} A)&=&\prod_{a\in\bar{A}}\mathbb{P}(X_a=0)\\
&=&\prod_{a\in\bar{A}}\exp\{-\Lambda(a)\}\\
&=&\exp\left\{-\sum_{a\in\bar{A}}\Lambda(a)\right\}.
\end{eqnarray*}
For each $n\geq1$, let ${\bf0}_n^{\cate}$ and ${\bf0}_n^{\catf}$ be the minimal element of $\cate_n$ and $\catf_n$ respectively, i.e.\ ${\bf0}_n^{\cate}={\bf0}_n^{\catf}=\emptyset$, and define $\mu_n([{\bf0}_n^{\cate},A]):=\mathbb{P}_n(X^*\leq_{\cate} A)=\exp\{-\sum_{a\subset[n]:a\notin A}\Lambda_n(a)\}$ to be the probability measure on $\mathcal{E}_n$ and $\nu_n([{\bf0}_n^{\catf},B]):=\mathbb{P}(\alpha(X^*)\leq_{\catf} B)=\exp\left\{-\sum_{b\in\bar{B}}\Lambda_n(b)\right\}$ the probability measure on $\catf_n$ induced by $\mu_n$ through $\alpha$, shown above, for some non-negative mean measure $\Lambda_n(\cdot)$ on $2^{[n]}$.  

For a subset $A\subset\mathbb{N}$, let $\#A$ denote the cardinality of $A$, i.e.\ the number of elements in $A$.  We now show that a necessary and sufficient condition for $\mu_n$, and hence $\nu_n$, to be infinitely exchangeable under action of $(\varphi_{m,n}^{\cate},m\leq n)$ is that for every $n\geq1$, $\Lambda_n(a)=\lambda_n(\#a)$ for some collection of non-negative real numbers $(\lambda_n(r),n\geq1,0\leq r\leq n)$ which satisfy
\begin{equation}\lambda_n(r)=\lambda_{n+1}(r)+\lambda_{n+1}(r+1).\label{eq:consistent lambda}\end{equation}
\begin{thm}\label{thm:inf exch lambda} The collection $(\mu_n,n\geq1)$ of probability measures on $\cate^{\rest}$ and $(\nu_n,n\geq1)$ on $\catf^{\rest}$ are infinitely exchangeable if and only if the collection of mean measures $(\Lambda_n,n\geq1)$ satisfy
\begin{itemize}
	\item[(a)] for every $n\geq1$, $\Lambda_n(a)=\lambda_n(\#a)$ for all $a\in\mathcal{F}_n$ for some collection of measures $(\lambda_n(\cdot),n\geq1)$ on the non-negative real numbers and
	\item[(b)] for every $n\geq1$, $\lambda_n(r)=\lambda_{n+1}(r)+\lambda_{n+1}(r+1)$ for $r=0,1,\ldots,n$.
\end{itemize}
Moreover, if (a) and (b) hold, the measure $(\mu^*_n,n\geq1)$ induced on $\catg^{\rest}$ by $(\mu_n,n\geq1)$ through $\beta\circ\alpha$, i.e.\ $\mu^*_n:=\mu_n\alpha^{-1}\beta^{-1}=\nu_n\beta^{-1}$, is infinitely exchangeable.
\end{thm}
\begin{proof}
Suppose (a) and (b) hold.  Let $n\geq1$ and suppose $E\in\mathcal{E}_n$.
Then $$\mu_n([{\bf0}^{\cate}_n,E])=\exp\left\{\sum_{e\subseteq[n]:e\notin E}\lambda_n(\#e)\right\}.$$
Clearly, the $\mu_n$ are finitely exchangeable for each $n\geq1$ as the measure depends only on the cardinality of the elements of $\{e\subset[n]:e\notin E\}$, which directly depends on the cardinality of the elements of $E$, which are invariant under permutations.  

As noted in section \ref{section:projective systems}, the pullback $\varphi^{\cate^{-1}}_{m,n}$ under the restriction maps takes intervals to intervals.  Hence, for $E=\{E_1,\ldots,E_k\}\in\cate_m$, the interval $[{\bf0}^{\cate}_m,E]\subseteq\mathcal{E}_m$ maps to the interval $[{\bf0}^{\cate}_n,\sup\varphi^{\cate^{-1}}_{m,n}(E)]\subseteq\cate_n$, where $\sup\varphi^{\cate^{-1}}_{m,n}(E)=E\cup\{m+1,\ldots,n\}$ is the unique maximal element of $\varphi^{\cate^{-1}}_{m,n}(E)$ in $\mathcal{E}_n$.  To simplify notation, write $\sup_{m,n} E:=\sup\varphi^{\cate^{-1}}_{m,n}(E)$ in what follows.

Hence, for consistency under sampling by restriction maps we must have
\begin{equation}\mu_n\left([{\bf0}_n,E]\right)=\mu_{n+1}\left([{\bf0}_{n+1},\sup_{n,n+1} E]\right)\label{eq:consistent mu}\end{equation}
Consistency follows from this since we now have
$$\mu_{n+1}([{\bf0}_{n+1},\sup_{m,n} E])=\exp\left\{\sum_{e\subseteq[n+1]:e\notin\sup_{m,n} E}\lambda_{n+1}(\#e)\right\},$$
which reduces \eqref{eq:consistent mu} to 
\begin{equation}\sum_{e\subseteq[n]}\lambda_n(\#e)=\sum_{e'\subseteq[n+1]:e'\notin\sup_{m,n} E}\lambda_{n+1}(\#e')\label{eq:consistent mu 2}\end{equation}
which is a sum over subsets of the power set of $[n]$ and $[n+1]$ respectively.  Under restriction, each $e\in 2^{[n]}$ corresponds to a two element subset of $[n+1]$, namely $\{e,e\cup\{n+1\}\}$.  Hence, for $\{e\subseteq[n]:e\notin E\}:=\{e_1,\ldots,e_k\}$, we have $\{e'\subseteq[n+1]:e'\notin{\sup_{m,n} E}\}=\{e_j,e_j\cup\{n+1\}:j=1,\ldots,k\}$ and \eqref{eq:consistent mu 2} is just
$$\sum_{e\subseteq[n]:e\notin E}\lambda_n(\#e)=\sum_{e\subseteq[n]:e\notin E}[\lambda_{n+1}(\#e)+\lambda_{n+1}(\#e+1)].$$  Infinite exchangeability for $\nu_n$ is endowed by $\mu_n$ through $\alpha$, as discussed in section \ref{section:infinite exchangeability}.

For the reverse implication, note that we start with the condition on the mean measures $\Lambda_n$
$$\sum_{e\subseteq[n]:e\notin E}\Lambda_n(e)=\sum_{e\subseteq[n]:e\notin E}[\Lambda_{n+1}(e)+\Lambda_{n+1}(e\cup\{n+1\})]$$
and exchangeability requires $\Lambda_n(a)=\Lambda_n(b)$ for all $a,b\subset[n]$ with $\#a=\#b$ so we can reduce this to the collection of measures $\lambda_n(\cdot)$ taking values in the non-negative real numbers, as we have done above.  And consistency requires of either $\mu_n$ or $\nu_n$ requires (b) to hold.

The infinite exchangeability of the measures $(\mu^*_n,n\geq1)$ on $\catg^{\rest}$ is a direct corollary of the infinite exchangeability of $\mu_n$ and the action of the functor $\beta\circ\alpha:\cate^{\rest}\rightarrow\catg^{\rest}$.
\end{proof}
\begin{cor}\label{thm:infinite measure}Suppose $(\lambda_n(r),n\geq1,r=0,\ldots,n)$ is a doubly indexed sequence of non-negative real numbers satisfying \eqref{eq:consistent lambda}, then there exists a measure $\mu^*$ on $\mathcal{G}$, the space of graphs with vertex set $\mathbb{N}$ such that
$$\mu_n^*(G)=\mu^*\left(\left\{G^*\in\mathcal{G}:G^*_{|[n]}=G\right\}\right),$$
where $G^*_{|[n]}$ denotes the restriction of $G^*$ to the vertex set $[n]$ and $\mu_n^*$ are the induced measures on $\mathcal{G}_n$ given above.\end{cor}
We make note of a correspondence between the solutions to \eqref{eq:consistent lambda} and the classical Hausdorff moment problem.  Connection between the Hausdorff moment problem and de Finetti's theorem have been shown by Diaconis and Freedman \cite{DiaconisFreedmanHausdorff} and are well known throughout the literature.  Some usable choices for $(\lambda_n(r),0\leq r\leq n)$ are
\begin{itemize}
	\item $\lambda_n(r)\propto\alpha^r(1-\alpha)^{n-r}$ for $0<\alpha<1$ and
	\item $\lambda_n(r)\propto {n\choose{r}} ^{-1}$.
\end{itemize}
\section{Cluster analysis}\label{section:clustering}
Given a measure $\mu_n^*$ on $\mathcal{G}_n$ which is based on a collection $(\lambda_n(r), n\geq1,r=0,1,\ldots,n)$ which satisfy \eqref{eq:consistent lambda}, we can easily calculate the marginal distribution that a triple of vertices, e.g.\ $i,j,k$, is transitive, i.e.\ $i\sim j,i\sim k,j\sim k$, given two are adjacent, e.g.\ $i\sim j$ and $i\sim k$, by a standard exchangeability argument.  In particular, for any $n\geq3$ and $i,j,k\in[n]$ we have 
\begin{eqnarray*}
\mu_n^*[\{i\sim j,i\sim k,j\sim k\}\vert\{i\sim j,i\sim k\}]&=&\mu^*_3[\{1\sim2,1\sim3,2\sim3\}\vert\{1\sim2,1\sim3\}]\\
&=&\frac{1-e^{-\lambda_3(3)}(1-(1-e^{-\lambda_3(2)})^3)}{1-e^{-\lambda_3(3)}(1-(1-e^{-\lambda_3(2)})^2)},
\end{eqnarray*}
which, unlike the Erd\"os-R\'enyi process, does not correspond to the clustering coefficient.  This calculation only accounts for the subgraph comprised of the vertices $i,j,k$ and any edges between them.  In general, if we observe a graph $G\in\catg_n$ and we are told that $i\sim j$ and $i\sim k$ but not told if there is an edge between $j$ and $k$ or not, but we also know the rest of the graph structure, this information is relevant for determining whether or not $j\sim k$ in $G$, and has implications in the inference of missing links in network data sets.

For example, consider the two graphs $G_1$ and $G_2$ below where
\begin{displaymath}
G_1=\bordermatrix{\text{}&1 & 2 & 3 & 4\cr
1 & - &  & 1 & 0\cr
2 &  & - & 1 & 0\cr
3 & 1 & 1 & - & 1\cr
4 & 0 & 0 & 1 & -}\end{displaymath}
and
\begin{displaymath}
G_2 = \bordermatrix{\text{}&1 & 2 & 3 & 4\cr
1 & - &  & 1 & 1\cr
2 &  & - & 1 & 1\cr
3 & 1 & 1 & - & 0\cr
4 & 1 & 1 & 0 & -},
\end{displaymath}
and the presence or absence of an edge between $1$ and $2$ is unknown, but the rest of this network is known.  Given what we do know about $G_1$, there are three possible monotone sets which correspond to $G_1$,
\begin{displaymath}
\begin{array}{ccc}
\langle\{1,2\},\{1,3\},\{2,3\},\{3,4\}\rangle,&\langle\{1,2,3\},\{3,4\}\rangle,&\langle\{1,3\},\{2,3\},\{3,4\}\rangle.
\end{array}
\end{displaymath}
On the other hand, there are five monotone sets corresponding to $G_2$,
\begin{displaymath}
\begin{array} {ccc}
\langle\{1,2,4\},\{1,2,3\}\rangle,&
\langle\{1,2,3\},\{1,3\},\{2,3\}\rangle&
\langle\{1,2,3\},\{1,4\},\{2,4\}\rangle\\
\langle\{1,2\},\{1,4\},\{2,4\},\{2,3\},\{1,3\}\rangle,&\langle\{1,3\},\{1,4\},\{2,3\},\{2,4\}\rangle.&
\end{array}
\end{displaymath}
 So the information given by the rest of the network, and any edges, or lack thereof, which involve any of the $i,j,k$ of interest, affects the conditional probability of $j\sim k$, e.g.\ $1\sim 2$ in this case.

  In general, the clustering of this process is expected to be larger than the marginal probability expression above as the presence of a cluster of three vertices increases the probability of other clusters which involve these vertices.  The nature of the construction, e.g.\ description of the functor $\beta\circ\alpha$, leads to overlapping of the various subsets of $A$ which is ``forgotten'' in the projection onto $\mathcal{G}_n$ and provides various ways by which clustering can occur.
  
  The construction also allows us to move between the space of graphs and that of monotone sets, which is helpful in calculations.

\subsection{Detecting clusters}\label{section:cluster analysis}
The nature of this construction naturally lends itself to methods in cluster analysis, which has been studied in certain applications in statistics and machine learning \cite{BleiJordanNg2004,McCullaghYang2006, McCullaghYang2008}.  The setting is as follows.  Let $n\geq1$ be the size of a sample for which we label statistical units, e.g.\ individuals, arbitrarily in $[n]$ and observe a network for this sample, i.e.\ an undirected graph $G\in\mathcal{G}_n$.  Along with $[n]$, let $\sim_1,\ldots,\sim_k$ be a collection of different equivalence relations on $[n]$.  A collection of labels $\{i_1,\ldots,i_m\}$ is said to form a {\em cluster}, or community, in our network if, for some $l=1,\ldots,k$, $i_p\sim_l i_q$ for every $p,q=1,\ldots,m$.  Inferring clusters in networks has implications, for example, in the problem of data deduplication and parsing for semi-structured text data sets as well as inferring communities and missing links in social networks.

In a statistical setting for social networks could represent different `types' of relationships among individuals.  That is, individuals in a social network are associated by certain relationships which underlie the network, e.g.\ a 4-node clique in a network could arise from 4 nodes belonging to the same cluster, or due to the overlap of two 3-node clusters and a 2-node cluster, e.g.\ $\{1,2,3\}, \{1,2,4\}$ and $\{3,4\}$, which both result in the presence of the clique $\{1,2,3,4\}$ in the projected network.  
\subsubsection{A statistical model}
In the setting of section \ref{section:construction} consider an infinite population $\mathcal{U}$ of units from which we sample a finite number $n\geq1$ which we label in $[n]$, i.e.\ our sample is $u_1,\ldots,u_n$, and we observe for this sample a network, or graph, $G\in\catg_n$ which we assume to have been generated according to the Poisson point process (PPP) construction on $2^{[n]}$ which we laid out above.  In particular, let $\Lambda:=(\Lambda_n(\cdot),n\geq1)$ be a family of mean measures on $(2^{[n]},n\geq1)$ such that $\Lambda_n(a)=\lambda_n(\#a)$ for every $a$ and the $\lambda_n(\cdot)$ satisfy \eqref{eq:consistent lambda}.  Let $\mu_n,\nu_n$ and $\mu^*_n$ be the measures on $\cate_n,\catf_n$ and $\catg_n$ defined in above sections.  Then we have the following model:
\begin{eqnarray*}
X&\sim&\mbox{PPP}(\Lambda)\\
G|\Lambda &\sim&\mu^*.
\end{eqnarray*}

  We now imagine reversing this process to infer whether a given complete subgraph $H\subset G$ of units $\{u_1,\ldots,u_h\}$ represents a cluster of $u_1,\ldots,u_n$.  This amounts to computing the conditional probability that $X_H>0$ given that $H\subset G$.

Suppose we observe a network $G\in\mathcal{G}_n$ with complete subgraph $H$.  Under the inverse image of the functor $\beta$, we have that $\beta^{-1}(G)\in\mathcal{F}_n$ is a collection of monotone sets which correspond to $G$.  Furthermore, the inverse image of the least monotone cover $\alpha^{-1}[\beta^{-1}(G)]$ is a collection of possible subsets of $2^{[n]}$ which have least monotone covers corresponding to $\beta^{-1}(G)$.  As we have shown in theorem \ref{thm:inf exch lambda}, the consistency condition in \eqref{eq:consistent lambda} guarantees infinite exchangeability of $(\mu_n)$ as well as $(\mu_n\alpha^{-1})$ and $(\mu_n(\beta\alpha)^{-1})$ on $\catf^{\rest}$ and $\catg^{\rest}$ respectively. 

Hence $(\beta\alpha)^{-1}(G)$ is the collection of random subsets in $\cate_n$ which correspond to $G$.  Define $Z:=Z(H,G):\subseteq (\beta\alpha)^{-1}(G)$ to be the elements of $(\beta\alpha)^{-1}(G)$ which contain the set $H$, and $\bar{Z}$ the complement of $Z$ in $(\beta\alpha)^{-1}(G)$, i.e.\ $\bar{Z}:=(\beta\alpha)^{-1}(G)-Z$.  Then the conditional probability that $H$ is a cluster in the subsample of $\mathcal{U}$ given $G$ is
\begin{equation}\pi(H;G):=\mathbb{P}(H\mbox{ a cluster})=\mu_n(Z)/\mu_n[(\beta\alpha)^{-1}(G)].\label{eq:cluster H}\end{equation}
The conditional probability calculation in \eqref{eq:cluster H} is valid for inference of a cluster $H$ if we are interested in the presence specifically of a cluster with exactly the elements of $H$ and not a cluster $H^*\supset H$ which contains the entire cluster $H$.  In applications, we might not be able to observe clusters at such a fine level.  In particular, if there is a cluster $H^*\supseteq H$ in the population, we may not be able to observe the presence or absence of a sub-cluster $H$ within the larger cluster.  In this case, it is reasonable to consider performing inference at a coarser level.  That is, for a given collection $H:=\{u_1,\ldots,u_h\}$, we wish to determine whether there is a cluster $H^*\supseteq H$ in the population.

Given a graph with clique $H$, let $H^*(H,G):\subseteq\beta^{-1}(G)$ be the monotone subset of $\beta^{-1}(G)$ such that $H\in H^*$.  In particular, 
$$H^*(H,G):=[H,{\bf1}_n]\cap\beta^{-1}(G).$$
Given $G$, the conditional probability of $H^*(H,G)$ is
$$\mathbb{P}(H^*(H,G)|G)=\nu_n\left[[H,{\bf1}_n]\cap \beta^{-1}(G)\right]/\nu_n[\beta^{-1}(G)].$$
\subsubsection{Stochastic classification}\label{section:stochastic classification}
Stochastic classification models have been studied previously by McCullagh and Yang \cite{McCullaghYang2008} in the context of the Gauss-Ewens clustering process.  In that model, a finite sample $u_1,\ldots,u_n$ is taken from an infinite population $\mathcal{U}$ of units for which we observe some feature $Y_i:=Y(u_i)\in\mathcal{S}$, for some subspace $\mathcal{S}$, usually $\mathcal{S}\subseteq\mathbb{R}^d$.  Associated to $\mathcal{U}$ is a partition $B$ of $\mathbb{N}$ and conditional on $B$ the vector $Y:=(Y_1,\ldots,Y_n)$ is normally distributed with mean and covariance which depend on $B$ only through the restriction $B_{|[n]}$ of $B$ to $[n]$.  For a newly sampled individual $u^*$, it is shown how to classify $u^*$ based on its feature $Y^*:=Y(u^*)$ and the data $(Y_1,\ldots,Y_n)$ and $B_{|[n]}$ already obtained.  This is carried out by computing the conditional distribution that $u^*$ belongs to each block $b\in B_{|[n]}$ or a possibly new block of $B$ given the data $Y_1,\ldots,Y_n,Y^*$ and $B_{|[n]}$.  The infinite exchangeability of the Gauss-Ewens process is a tool which allows the computation of conditional distributions in this setting.

In our setting, we also assume an infinite population $\mathcal{U}$ and an associated network (undirected graph) $G$ of $\mathcal{U}$ which is generated by the Poisson point process recipe above.  Suppose we sample $u_1,\ldots,u_n$ from $\mathcal{U}$ and we observe the restriction $G_{|[n]}$ of $G$ to $[n]$ as well as the component of $\cate_n$ which corresponds to $G_{|[n]}$.  That is, we assume we observe the realization of $X^*$, the projection of the Poisson point process on $2^{[n]}$ onto $\cate_n$.  For a new individual $u^*\in\mathcal{U}$, suppose we observe its connections within the network, i.e.\ we have complete knowledge of $G_{|[n+1]}$, but nothing about its associated clusters in $\cate_{n+1}$.

Infinite exchangeability of the measures $(\mu_n,n\geq1)$ on $\cate^{\rest}$ makes the specification of the conditional distribution of various outcomes straightforward.  Given $X^*\in\cate_n$ and $\tilde{G}\in\catg_{n+1}$, let $H^*\in\cate_{n+1}$ be a collection which is consistent with $X^*$ and $\tilde{G}$, i.e.\ $H^*\subseteq(\beta\alpha)^{-1}(\tilde{G})\cap\varphi^{\cate^{-1}}_{n,n+1}(X^*)$, we have  
\begin{eqnarray*}
\pr(H^*|\tilde{G},X^*)&=&\frac{\mu_{n+1}(H^*)}{\mu_{n+1}[(\beta\alpha)^{-1}(\tilde{G})\cap\varphi_{n,n+1}^{\cate^{-1}}(X^*)]}.
\end{eqnarray*}

\end{document}